\documentclass[12pt]{amsart}

\usepackage[left=1in,top=1.25in,right=1in,foot=1.25in]{geometry}
\usepackage{appendix}
\usepackage{amsmath}
\usepackage{amsfonts,amssymb}
\usepackage{sidecap}
\usepackage{graphicx}
\usepackage{subfig}
\usepackage{amsthm}
\usepackage{dsfont}
\usepackage{enumerate}
\usepackage{color}
\usepackage{graphicx}
\usepackage{comment}
\newtheorem{thm}{Theorem}[subsection]
\newtheorem{cor}{Corollary}[subsection]
\newtheorem*{thmnonum}{Theorem}

\newtheorem*{cornonum}{Corollary}
\theoremstyle{definition}

\theoremstyle{remark}

\theoremstyle{remark}
\newtheorem{rmk}{Remark}[subsection]
\theoremstyle{plain}
\newtheorem{lem}[thm]{Lemma}
\newcommand{\la}[0]{\langle}
\newcommand{\ra}[0]{\rangle}

\newcommand{\Ric}[0]{\text{Ric}}

\newcommand{\Ind}[0]{\text{Ind}}
\title{Index bounds for free boundary minimal surfaces of convex bodies}
\author{Pam Sargent}

\begin{document}
\thanks{
This work is partially supported by NSERC}

\maketitle
\begin{abstract} In this paper, we give a relationship between the eigenvalues of the Hodge Laplacian and the eigenvalues of the Jacobi operator for a free boundary minimal hypersurface of a Euclidean convex body. We then use this relationship to obtain new index bounds for such minimal hypersurfaces in terms of their topology. In particular, we show that the index of a free boundary minimal surface in a convex domain in $\mathbb{R}^3$ tends to infinity as its genus or the number of boundary components tends to infinity.
\end{abstract}

\section{Introduction}
In this paper we look at the problem of obtaining lower bounds on the index of free boundary minimal surfaces of convex bodies in terms of their topology. Index estimates for minimal surfaces are generally difficult to obtain, and there are few minimal surfaces for which the index is explicitly known. However, index bounds can help in the classification of minimal surfaces, especially when the topology is explicitly represented in the bounds, and have applications in understanding the relationships between the curvature and topology of Riemannian manifolds. Moreover, minimal surfaces whose index is known have proven to be useful in other problems; Urbano's \cite{Ur90} index characterization of the Clifford torus as being the unique minimal surface in $\mathbb{S}^3$ of index 5 was recently used by Marques and Neves \cite{Ma12} in their celebrated proof of the longstanding Willmore Conjecture. In \cite{Sav10}, Savo was able to obtain index bounds for minimal hypersurfaces in $\mathbb{S}^n$ in terms of their topology making use of the Laplacian on 1-forms. His work has inspired the approach taken in this paper.

\subsection{Free Boundary Minimal Hypersurfaces in Convex Bodies}
A submanifold $M$ of a compact Riemannian manifold $\overline{M}$ with boundary $\partial M\subset \partial \overline{M}$ is said to be a \emph{free boundary minimal submanifold} in $\overline{M}$ if it is a critical point for the volume functional among submanifolds with boundary in $\partial \overline{M}$. That is, $M$ is a free boundary minimal submanifold of $\overline{M}$ if for every admissible variation $M_t$ of $M$, $\frac{d}{dt}\text{Vol}(M_t)\big|_{t=0} =0$. 
The first variation formula for a variation $M_t$ of $M$ with variation field $V$ is given by,
\begin{equation*}
\frac{d}{dt}\text{Vol}(M_t)\big|_{t=0} = -\int_M\la{V, H}\ra dV +\int_{\partial M}\la{V, \eta}\ra dA,
\end{equation*}
where $\eta$ is the outward unit conormal vector field.
It follows that $M$ is a free boundary minimal submanifold of $\overline{M}$ if and only if $H\equiv 0$ and $\eta$ is orthogonal to $T(\partial \overline{M})$, \textit{i.e.}, $M$ meets $\partial \overline{M}$ orthogonally.

Throughout, we will focus our attention on free boundary minimal hypersurfaces $M^n$ properly immersed in convex bodies $B^{n+1}$. Here, a \emph{convex body} is a smooth, compact, connected $(n+1)$-dimensional submanifold of $\mathbb{R}^{n+1}$ for which the scalar second fundamental form of the boundary satisfies $h^{\partial B}(U,U)<0$ (with respect to the outward pointing normal vector) for all vectors $U$ tangent to $\partial B$.

We will also place some attention on the special case when $B=\mathbb{B}$, the Euclidean ball, as there are more existence results for free boundary minimal hypersurfaces of Euclidean balls. Free boundary minimal hypersurfaces of Euclidean balls have also been shown to have an alternative characterization: in \cite{FS11}, Fraser and Schoen showed that if $\Sigma^k$ is a properly immersed submanifold of the Euclidean unit ball $\mathbb{B}^{n+1}$, then $\Sigma$ is a free boundary minimal submanifold if and only if the coordinate functions of the immersion are (Steklov) eigenfunctions of the Dirichlet-to-Neumann map with (Steklov) eigenvalue 1. Furthermore, free boundary minimal surfaces in $\mathbb{B}^{n+1}$ are extremal surfaces for the Steklov eigenvalue problem.

\subsection{The Index of a Minimal Hypersurface}
Suppose that $M^n\subset B^{n+1}$ is a free boundary minimal hypersurface and that $N$ is a smooth unit normal vector field. Then, for a normal variation with variation field $uN$, the second variation formula is
\begin{equation*}
\frac{d^2}{dt^2}\text{Vol}(M_t)\big|_{t=0} = \int_M \left(\|\nabla u\|^2 - \|A\|^2u^2\right) dV + \int_{\partial M}h^{\partial B}(N,N) u^2 dA.
\end{equation*}

Let $J$ denote the \emph{Jacobi operator} (also called the \emph{stability operator}),
\begin{equation*}
J = \Delta - \|A\|^2,
\end{equation*}
and let $Q$ denote the associated symmetric bilinear form,
\begin{equation*}
\begin{split}
Q(u) &= \int_M \left[ \|\nabla u\|^2- \|A\|^2u^2\right]dV+\int_{\partial M}h^{\partial B}(N,N) u^2\, dA\\
& = -\int_M u\cdot Ju\, dV + \int_{\partial M} \left(\frac{\partial u}{\partial \eta}+h^{\partial B}(N,N)u\right)u\, dA.
\end{split}
\end{equation*}
We say that $\lambda(J)$ is an eigenvalue of $J$ with eigenfunction $u\in C^{\infty}(M)$ if 
\begin{equation*}
\begin{cases}
Ju = \lambda u \ \ \text{on } M,\\
\frac{\partial u}{\partial \eta} +h^{\partial B}(N,N) u =0 \ \ \text{on } \partial M.
\end{cases}
\end{equation*}

The (\emph{Morse}) \emph{index} of a minimal hypersurface is the maximal dimension of a subspace of $C^{\infty}(M)$ on which the second variation is negative.

A free boundary minimal hypersurface is said to be \emph{stable} if it has index 0. For free boundary minimal hypersurfaces in $B^{n+1}$, there are none which are stable. This is easy to see since if we use the variation with variation field $1\cdot N$, then we get that
\begin{equation*}
Q(1) = -\int_M \|A\|^2\, dV + \int_{\partial M}\left(0+h^{\partial B}(N,N)\right)\cdot 1 \,dA <0.
\end{equation*}
Hence, any free boundary minimal hypersurface in $B^{n+1}$ has index at least 1.

\subsection{Examples and Existence Results} For general convex bodies, little is known about the existence of free boundary minimal submanifolds.  In \cite{St84}, Struwe showed the existence of a (possibly branched) immersed free boundary minimal disk in any domain in $\mathbb{R}^3$ diffeomorphic to $\mathbb{B}^3$, and Gr\"uter and Jost \cite{GJ86} showed that there is an embedded free boundary minimal disk in any convex body in $\mathbb{R}^3$. Jost \cite{Jo89} was also able to show that any convex body in $\mathbb{R}^3$ actually contains at least three embedded free boundary minimal disks.  More recently, Maximo, Nunes and Smith \cite{MNS} showed that any convex body in $\mathbb{R}^3$ contains a minimal annulus. By the above argument, we know that any free boundary minimal hypersurface of a convex body has index at least one. Currently, there are no other existence results or index results for minimal surfaces of greater topological complexity, nor are there any existence results in higher dimensions.  

If we focus on minimal submanifolds of Euclidean balls, then more is known. 
The simplest examples of free boundary minimal submanifolds in $\mathbb{B}^{n+1}$ are the equatorial $k$-planes $D^k \subset \mathbb{B}^{n+1}$. By \cite{Ni85} and \cite{FS15b}, any simply connected free boundary minimal surface in $\mathbb{B}^n$ must be a flat equatorial disk,  and it is well known that the equatorial disk has index 1 (see p.\ 3741 in \cite{Fr07}). In fact, it is the only free boundary minimal surface of $\mathbb{B}^{3}$ to have index 1. However, there are many examples of free boundary minimal surfaces of different topological type. The critical catenoid, a minimal surface with genus 0 and 2 boundary components, is an explicit example of such a surface. In \cite{FS15}, Fraser and Schoen prove existence of free boundary minimal surfaces in $\mathbb{B}^3$ with genus 0 and $k>0$ boundary components. Using gluing techniques, in \cite{FPZ15} Folha, Pacard and Zolotareva give an independent construction of free boundary minimal surfaces in $\mathbb{B}^3$ of genus 0 with $k$ boundary components for $k$ large. They are also able to use the same techniques to construct a genus 1 free boundary minimal surface with $k$ boundary components for $k$ large. Kapouleas and M.\ Li have constructed free boundary minimal surfaces in $\mathbb{B}^3$ with any sufficiently large genus and 3 boundary components by gluing an equatorial disk to a critical catenoid, though this has not yet been published. If $M$ is not an equatorial disk, then all that is known about its index is that $\Ind(M)\geq 3$ (see Theorem 3.1 in \cite{FS15}).

In this paper, we give a relationship between the eigenvalues of the Jacobi operator and the eigenvalues of the Laplacian on 1-forms and, as a corollary, obtain new index bounds for orientable free boundary minimal hypersurfaces of convex bodies. More specifically, our first main result is:
\begin{thm}\label{ER} Let $M^n$ be an orientable free boundary minimal hypersurface of a convex body in $\mathbb{R}^{n+1}$ with Jacobi operator $J$. Then, for all positive integers $j$, one has that
\begin{equation*}
\lambda_j(J)\leq\lambda_{m(j)}(\Delta_1),
\end{equation*}
where $m(j) = {n+1\choose 2}(j-1)+1$ and $\lambda_{m(j)}(\Delta_1)$ is the $m(j)$th eigenvalue of the Laplacian eigenvalue problem with absolute boundary conditions.
\end{thm}

Let $\beta^1_a=\dim H_a^1(M)$ be the first absolute Betti number of $M$. Our second main result is:
\begin{thm}{\upshape (Index Bound)}\label{IB} If $M$ is an orientable free boundary minimal hypersurface of a convex body in $\mathbb{R}^{n+1}$, then
{\upshape\begin{equation*}
\Ind(M) \geq \left \lfloor{\frac{\beta_a^1+{n+1\choose 2}-1}{{n+1\choose 2}}}\right \rfloor.
\end{equation*}}
\end{thm}

\begin{cor}\label{ICor} If $M$ is an orientable free boundary minimal surface of a convex body in $\mathbb{R}^3$ with genus $g$ and $k$ boundary components, then
{\upshape\begin{equation*}
\Ind(M)\geq \left \lfloor{\frac{2g+k+1}{3}}\right \rfloor.
\end{equation*}} 
\end{cor}

Corollary \ref{ICor} provides new index bounds for free boundary minimal surfaces of $\mathbb{B}^3$ with large topology. In particular, it shows that $\text{Ind}(M)>3$ when $2g+k\geq 11$ and $\text{Ind}(M)$ tends to infinity as the genus or the number of boundary components tend to infinity.

The paper is structured as follows: In the second section, we outline the basic notation and conventions that we will use throughout the paper, and  give a brief introduction to the Hodge Laplacian on $p$-forms. Here, we define the Hodge Laplacian on $p$-forms and then focus on the special case when $p=1$. 
We also introduce the two main boundary conditions for the eigenvalue problem of the Laplacian for 1-forms on a manifold with boundary. 
In the third section, we provide several preliminary calculations that will ultimately allow us to see how the Jacobi operator acts on specific test functions, which will be needed to prove our main results.
We give the proofs of our two main results in the fourth section.

\vspace{1cm}

\section{Notation and Conventions}
Let $M^n$ be an orientable free boundary minimally immersed hypersurface in $B^{n+1}$ ($\partial M\neq \emptyset$). Throughout, we will let $N$ be a unit normal vector field on $M$. 

Let $D$ denote the Levi-Civita connection on $\mathbb{R}^{n+1}$ and $\nabla$ the Levi-Civita connection on $M$. We will let $A$ denote the second fundamental form of $M\subset B$, and $S$ the associated shape operator. That is, for $X,Y\in \Gamma(TM)$,
\begin{equation*}
\begin{split}
A(X,Y) &= \left(D_X Y\right)^N = \la{D_X Y,N}\ra\cdot N\\
S(X) & = -\left(D_XN\right)^T,
\end{split}
\end{equation*}
so that $\la{A(X,Y),N}\ra = \la{S(X),Y}\ra$

For any parallel vector field $\overline{V}$ in $\mathbb{R}^{n+1}$, we have the orthogonal decomposition
\begin{equation*}
\overline{V} = V+V^{N},
\end{equation*}
where $V\in TM$ is the orthogonal projection of $\overline{V}$ onto M and $V^{N} = \la{\overline{V},N}\ra\cdot N\in NM$. Since parallel vector fields on $\mathbb{R}^{n+1}$ and their orthogonal projections onto $M$ will be used throughout, we introduce the following vector spaces:
\begin{equation*}
\begin{split}
\overline{\mathcal{P}} &= \{\text{parallel vector fields on } \mathbb{R}^{n+1}\},\\
\mathcal{P} & = \{\text{vector fields on } M \text{ which are orthogonal projections of elements of } \overline{\mathcal{P}}\}.
\end{split}
\end{equation*}

Throughout, we will let $\Delta_p$ denote the Hodge Laplacian on $p$-forms (though the $p$ will usually be dropped for convenience) and we will  let $\nabla^*\nabla$ denote the rough Laplacian on vector fields. So, if $\omega$ is a $p$-form on $M$ and $\xi$ is a vector field on $M$, then
\begin{equation*}
\begin{split}
\Delta_p \omega & = (d\delta + \delta d)\omega\\
\nabla^*\nabla\xi & =-\sum_{j=1}^n\left(\nabla_{e_j}\nabla_{e_j}\xi-\nabla_{\nabla_{e_j}e_j}\xi\right),
\end{split}
\end{equation*}
where $d$ is the exterior derivative, $\delta$ is the codifferential, and $\{e_1,\ldots,e_n\}$ is any local orthonormal frame of $TM$.  Recall that a vector field $X$ is dual to a 1-form $\theta$ if and only if $\la{X,Y}\ra = \theta(Y)$ for all $Y\in\Gamma(TM)$. 
If $\xi$ is the vector field dual to $\omega$, then one can also define the Hodge Laplacian of $\xi$, denoted $\Delta\xi$, to be the vector field dual to the 1-form $\Delta_1\omega$.
The Bochner formula relates the two Laplacians:
\begin{equation*}
\Delta\xi = \nabla^*\nabla\xi +\Ric(\xi),
\end{equation*}
where $\Ric$ is seen as a symmetric endomorphism of $TM$.

To get a bound on the index of $M$, we will consider the following eigenvalue problem defined by the \emph{absolute} boundary conditions:
\begin{equation*}\label{AbsCond}
\begin{cases}
J_1\omega & =\lambda\omega,\\
i^*\iota_{\eta}\omega &= i^*\iota_{\eta}d\omega =0,
\end{cases}
\end{equation*}
where $i$ is the inclusion $\partial M\hookrightarrow M$,  $\iota_{\eta}$ denotes interior multiplication by $\eta$ and $J_1$ is the Jacobi operator on 1-forms defined by $J_1 = \Delta_1-\|A\|^2$. We will often drop the subscripts for convenience. These absolute boundary conditions are a generalization of Neumann boundary conditions for functions.
 We say that $\omega$ is \emph{tangential} on $\partial M$ if $i^*\iota_{\eta}\omega =0$, \textit{i.e.}, $\omega$ vanishes whenever its argument is normal to the boundary of $M$. So, if $\omega$ satisfies the absolute boundary conditions, then both $\omega$ and $d\omega$ are tangential ($d\omega$ is tangential whenever one of its arguments is normal to $\partial M$).

We define the following space of harmonic 1-forms
\begin{equation*}
\begin{split}
\mathcal{H}_N^1(M)& = \{\omega\in\Omega^1(M) \ | \ \Delta\omega =0, \ \omega \ \text{satisfies the absolute boundary conditions}\},\\
\end{split}
\end{equation*}
and note that $\beta_a^1 = \dim H^1_a(M) = \dim \mathcal{H}_N^1(M)$, where $H^1_a(M)$ is the first absolute cohomology group of $M$.

\vspace{1cm}

\section{Preliminary Calculations}
The calculations done here are analogous to those done by Savo in \cite{Sav10} for the case of a minimal hypersurface in $\mathbb{S}^{n+1}$. In $\mathbb{S}^{n+1}$, a hypersurface has two normal vectors (one tangent to the sphere and one normal to both the sphere and the hypersurface) whereas a free-boundary minimal hypersurface of a convex body $B^n$ just has one. The absence of a second normal vector simplifies many of the preliminary calculations. 
\begin{lem}\label{PPC} Let $\overline{V}\in\overline{\mathcal{P}}$ and let $V\in\mathcal{P}$ be its orthogonal projection onto $M$. Let $A$ and $S$ be the second fundamental form and shape operator (respectively) of the immersion $\phi:M\rightarrow B^n$. Then{\upshape
\begin{enumerate}[(a)]
\item $\nabla\la{\overline{V},N}\ra = -S(V)$.
\item $\Delta\la{\overline{V},N}\ra = |S|^2\la{\overline{V},N}\ra$.
\end{enumerate}}
\end{lem}
\begin{proof}
To show (a), take any $X\in\Gamma(TM)$. Then we have that
\begin{equation*}
\la{\nabla\la{\overline{V},N}\ra,X}\ra  = \la{D_X\overline{V}, N}\ra + \la{\overline{V}, D_X N}\ra = \la{\overline{V}, D_X N}\ra,
\end{equation*}
since $\overline{V}$ is parallel. Now, since $\la{N, D_X N}\ra = \frac{1}{2}X\left(\|N\|^2\right) \equiv 0$ and $[X,V]$ is tangent to $M$, we have that
\begin{equation*}
\la{\overline{V}, D_X N}\ra  = -\la{D_X V, N}\ra = -\la{D_V X, N}\ra = \la{X, \left(D_V N\right)^T}\ra.
\end{equation*}
Hence, $\nabla\la{\overline{V}, N}\ra = -S(V)$.

For (b), let $\{e_1, \ldots e_n\}$ denote normal coordinate vector fields centred at a point $p\in M$. Then (at $p$),
\begin{equation*}
\begin{split}
-\Delta\la{\overline{V}, N}\ra & = \sum_{i=1}^n \la{\nabla_{e_i}\nabla\la{\overline{V}, N}\ra, e_i}\ra = \sum_{i=1}^n \la{\nabla_{e_i}\left(D_V N\right)^T, e_i}\ra = -\sum_{i=1}^n e_i\la{N, D_V e_i}\ra\\
& = -\sum_{i=1}^n \la{D_{e_i} N, A(V, e_i)}\ra + \la{N, D_{e_i}A(V, e_i)}\ra.
\end{split}
\end{equation*}
Since $D_{e_i} N$ has no normal component, and $A$ is symmetric, we have that
\begin{equation*}
\Delta\la{\overline{V}, N}\ra = \sum_{i=1}^n\la{N, \left(D_{e_i}A(e_i, V)\right)^N}\ra.
\end{equation*}
Now, it follows from the Codazzi Equation that
\begin{equation*}
\left(D_{e_i}A(e_i, V)\right)^N = \left(D_VA(e_i, e_i)\right)^N-2A(e_i, \nabla_V e_i) +A(\nabla_{e_i}V, e_i).
\end{equation*}
However
$A(e_i, \nabla_V{e_i}) = \left(D_{e_i}\nabla_V e_i\right)^N$,
and, at $p$,
$\la{D_{e_i}\nabla_V e_i, N}\ra  = -\la{\nabla_V e_i, D_{e_i} N}\ra= 0$.
Moreover, since $M$ is minimal,
$\sum_{i=1}^n\left(D_VA(e_i, e_i)\right)^N = 0$.
Therefore,
\begin{equation*}
\Delta\la{\overline{V}, N}\ra  = \sum_{i=1}^n\la{N, A(\nabla_{e_i}V, e_i)}\ra = -\sum_{i=1}^n\la{\left(D_{e_i}N\right)^T, \nabla_{e_i}V}\ra.
\end{equation*}
\end{proof}

\begin{lem}\label{PC1} For any vector field $\xi\in\Gamma(TM)$ and any $\overline{V}\in\overline{P}$ with orthogonal projection $V$,{\upshape
\begin{enumerate}[(a)]
\item $\Delta\xi = \nabla^*\nabla\xi-S^2(\xi)$.
\item $\nabla^*\nabla V = S^2(V), \ \Delta V =0$.
\end{enumerate}
}
\end{lem}

\begin{proof} Let $\{e_1,\ldots, e_n\}$ be local normal coordinate vector fields centred at a point $p\in M$. Then,
using the minimality of $M$ and the Gauss equation, we have that (at $p$)
\begin{equation*}
\text{Ric}(\xi) =\sum_{i,k=1}^nR_M(e_k,e_i, \xi, e_k)e_i  = -\sum_{i,k=1}^n \la{A(e_k, \xi), A(e_i, e_k)}\ra e_i.
\end{equation*}
Now 
\begin{equation*}
\begin{split}
-\sum_{k=1}^n \la{A(e_k, \xi), A(e_i, e_k)}\ra = -\sum_{k=1}^n \la{e_k, D_{\xi}N}\ra\la{e_k, D_{e_i}N}\ra & = - \la{(D_{\xi}N)^T, D_{e_i}N}\ra\\
& = \la{D_{(D_{\xi}N)^T}e_i+[e_i, (D_{\xi}N)^T], N}\ra\\
& = -\la{e_i, S^2(\xi)}\ra.
\end{split}
\end{equation*}
Therefore, $\text{Ric}(\xi) = -S^2(\xi)$, and (a) follows from the Bochner formula. 

To see that $\nabla^*\nabla V=S^2(V)$, we'll first show that $\nabla^*\nabla N=0$ in the sense that if $\{e_1,\ldots, e_n\}$ are again local normal coordinate vector fields centred at $p\in M$, then, at $p$,
\begin{equation*}
\sum_{i=1}^n\left(D_{e_i}(D_{e_i}N)^T\right)^T=0.\label{DN0}
\end{equation*}
Since, $D_{e_i}N$ is tangential,
\begin{equation*}
\begin{split}
\sum_{i=1}^n\left(D_{e_i}(D_{e_i}N)^T\right)^T  = \sum_{i,j=1}^n\la{D_{e_i}D_{e_i}N, e_j}\ra e_j
& = -\sum_{i,j=1}^ne_i\la{N, D_{e_i}e_j}\ra e_j\\
& = -\sum_{i,j=1}^n\la{N, D_{e_i}(A(e_j,e_i))}\ra e_j.
\end{split}
\end{equation*}
Now, it follows from the Codazzi Equation that $\left(D_{e_i}(A(e_j,e_i))\right)^N = \left(D_{e_j}(A(e_i,e_i))\right)^N$.
So, again using the minimality of $M$, we have that
\begin{equation*}
\sum_{i=1}^n\left(D_{e_i}(D_{e_i}N)^T\right)^T =-\sum_{i,j=1}^n\la{N, D_{e_i}(A(e_j,e_i))}\ra e_j = -\sum_{i,j=1}^n\la{N, D_{e_j}(A(e_i,e_i))}\ra e_j = 0.
\end{equation*}

Now, if we write $V=\overline{V} - \la{\overline{V}, N}\ra N$, then we can use this calculation and the fact that $\overline{V}$ is parallel to help us calculate $\nabla^*\nabla V$. 
\begin{equation*}
\begin{split}
\nabla^*\nabla V &= \sum_{i=1}^n \left(D_{e_i}(D_{e_i}\overline V)^T\right)^T - \left(D_{e_i}(D_{e_i}(\la{\overline{V}, N}\ra N))^T\right)^T\\
& = -\sum_{i=1}^n\left(D_{e_i}(\la{\overline{V}, N}\ra D_{e_i}N)\right)^T = -\left(\sum_{i=1}^ne_i(\la{\overline{V}, N}\ra) D_{e_i}N\right)-\la{\overline{V},N}\ra \nabla^*\nabla N\\
& = -\sum_{i=1}^n\la{V, D_{e_i}N}\ra D_{e_i}N = -\sum_{i=1}^n\la{e_i,D_V N}\ra D_{e_i}N = S^2(V).
\end{split}
\end{equation*}
The fact that $\Delta V = 0$ now follows from (a).
\end{proof}

\begin{lem}\label{LapIP} Let $\overline{V},\overline{W}\in\overline{\mathcal{P}}$ and let $V, W\in\mathcal{P}$ be their orthogonal projections onto $M$. Then, for any $\xi\in\Gamma(TM)$,{\upshape
\begin{enumerate}[(a)]
\item $\Delta\la{V,\xi}\ra = 2\la{S(V),S(\xi)}\ra + \la{V,\Delta\xi}\ra -2\la{\overline{V}, N}\ra\la{S,\nabla\xi}\ra$.
\item $\la{\nabla\la{\overline{V},N}\ra,\nabla\la{W,\xi}\ra}\ra = -\la{\overline{W},N}\ra\la{S(V),S(\xi)}\ra-\la{W,\nabla_{S(V)}\xi}\ra$.
\item $\Delta(\la{\overline{V},N}\ra\la{W, \xi}\ra) = |S|^2 \la{\overline{V}, N}\ra\la{W, \xi}\ra +2(\la{\overline{W}, N}\ra\la{S(V), S(\xi)}\ra+\la{W, \nabla_{S(V)}\xi}\ra)\\
\textcolor{white}{.} \qquad  \qquad\qquad  \qquad + \la{\overline{V}, N}\ra(2\la{S(W), S(\xi)}\ra+\la{W, \Delta\xi}\ra-2\la{\overline{W}, N}\ra\la{S, \nabla\xi}\ra)$.
\end{enumerate}
}
\end{lem}

\begin{proof}
Let $\{e_1,\ldots, e_n\}$ be local normal coordinate vector fields centred at a point $p\in M$. Then, at $p$,
\begin{equation*}
\Delta\la{V, \xi}\ra = \la{\nabla^*\nabla V, \xi}\ra -2\la{\nabla V, \nabla \xi}\ra + \la{V, \nabla^*\nabla \xi}\ra,
\end{equation*}
where $\la{\nabla V, \nabla \xi}\ra = \sum_{i=1}^n\la{\nabla_{e_i}V, \nabla_{e_i}\xi}\ra$.
From Lemma \ref{PC1} we have that $\nabla^*\nabla V = S^2(V)$ and $\nabla^*\nabla\xi = \Delta\xi+S^2(\xi)$. We also have that
\begin{equation*}
\la{S^2(V),\xi}\ra  = -\la{N, D_{D_V N}\xi}\ra = \la{D_{\xi}N, D_V N}\ra = \la{S(\xi), S(V)}\ra,
\end{equation*}
and, similarly, $\la{V, S^2(\xi)}\ra = \la{S(V), S(\xi)}\ra$. Therefore,
\begin{equation*}
\Delta\la{V, \xi}\ra = 2\la{S(V), S(\xi)}\ra+\la{V, \Delta\xi}\ra -2\la{\nabla V, \nabla \xi}\ra.
\end{equation*}
Finally,
\begin{equation*}
\la{\nabla_{e_i} V, \nabla_{e_i}\xi}\ra = \la{D_{e_i}(\overline{V}-\la{\overline{V}, N}\ra N), \nabla_{e_i}\xi}\ra = \la{\overline{V}, N}\ra\la{S(e_i), \nabla_{e_i}\xi}\ra.
\end{equation*}
Hence, summing over $i$ gives us that
\begin{equation*}
\Delta\la{V, \xi}\ra = 2\la{S(V), S(\xi)}\ra+\la{V, \Delta\xi}\ra -2\la{\overline{V}, N}\ra\la{S, \nabla\xi}\ra.
\end{equation*}

From Lemma \ref{PPC}(a) we know that $\nabla\la{\overline{V}, N}\ra = -S(V)$, so we just need to calculate $\nabla\la{W, \xi}\ra$. First, notice that for any vector field $X$ on $M$, since $\overline{W}$ is parallel,
\begin{equation*}
\nabla_X W  = \left(D_X(\overline{W} - \la{\overline{W}, N}\ra N)\right)^T\ = \la{\overline{W}, N}\ra S(X)
\end{equation*}
Hence, 
\begin{equation*}
\la{\nabla\la{W, \xi}\ra, X}\ra  =X(\la{W, \xi}\ra) = \la{\overline{W}, N}\ra\la{S(X), \xi}\ra + \la{W, \nabla_X\xi}\ra.
\end{equation*}
So, for $X=-S(V) \ (=\nabla\la{\overline{V}, N}\ra)$, we have that
\begin{equation*}
\begin{split}
\la{\nabla\la{\overline{V}, N}\ra, \nabla\la{W, \xi}\ra}\ra & = -\la{\overline{W}, N}\ra\la{S^2(V), \xi}\ra  -\la{W, \nabla_{S(V)}\xi}\ra\\
& = -\la{\overline{W}, N}\ra\la{S(V), S(\xi)}\ra  -\la{W, \nabla_{S(V)}\xi}\ra.\\
\end{split}
\end{equation*}
Now (c) follows from (a) and (b) and Lemma \ref{PPC}(b).
\end{proof}

Let $\overline{\mathcal{U}} = \{\overline{V}\in\overline{\mathcal{P}} \ | \ \|V\|\equiv 1\}$. Then $\mathcal{U}$ can naturally be identified with $S^n$ if we endow it with the measure $\mu = \frac{n+1}{\text{Vol}(S^n)}dv_{S^n}$. 
\begin{lem}\label{IC} For any $\overline{X}, \overline{Y}\in\mathbb{R}^{n+1}$, 
\begin{equation*}
\int_{\overline{\mathcal{U}}} \la{\overline{V}, \overline{X}}\ra\la{\overline{V}, \overline{Y}}\ra\, d\overline{V} = \la{\overline{X}, \overline{Y}}\ra.
\end{equation*}
\end{lem}
The proof of Lemma \ref{IC} follows from a direct, but tedious, calculation after changing to spherical coordinates and repeatedly applying the integral identity
\begin{equation*}
\int \sin^mx\, dx = -\frac{1}{m}\sin^{m-1}x\cos x + \frac{m-1}{m}\int \sin^{m-2}x\, dx.
\end{equation*}

\bigskip

The following lemma was originally proved by Ros \cite{Ro09} for free boundary minimal surfaces in a smooth domain in $\mathbb{R}^3$. Here, we extend his proof to obtain the analogous result for free boundary minimal hypersurfaces in smooth domains in $\mathbb{R}^n$.
\begin{lem}\label{BC}
Suppose $\xi$ is a vector field on $M$ dual to a 1-form $\omega$ which satisfies the absolute boundary conditions. Then, at a point $p\in\partial M$,
\begin{equation*}
\la{\nabla_{\eta}\xi, \xi}\ra =h^{\partial B}(N,N)\|\xi\|^2.
\end{equation*}
\end{lem}

\begin{proof}
Let $\eta$ be the (outward pointing) conormal vector along $\partial M$. Then, since $\omega$ satisfies the absolute boundary conditions on $\partial M$, at $p$ we have that
\begin{equation*}
\begin{split}
\omega(\eta) &= 0,\\
d\omega(\eta,t) & = \eta(\omega(t))-t(\omega(\eta))-\omega([\eta,t]) = 0,
\end{split}
\end{equation*}
for any vector $t\in T_p(\partial M)$. In particular, if $\xi$ is the vector field dual to $\omega$, the the first condition implies that $\xi_p\in T_p(\partial M)$, and so the second condition implies that $d\omega (\eta, \xi)=0$ at $p$. Now,
\begin{equation*}
\la{\xi, \nabla_\eta\xi}\ra = \eta\la{\xi, \xi}\ra - \la{\xi, \nabla_\eta\xi}\ra = (\nabla_{\eta}\omega)(\xi),
\end{equation*}
and we claim that $(\nabla_\eta\omega)(\xi) = (\nabla_{\xi}\omega)(\eta)$. To see this, note that, by definition,
\begin{equation*}
(\nabla_{\xi}\omega)(\eta) - (\nabla_{\eta}\omega)({\xi}) = {\xi}(\omega(\eta))-\omega(\nabla_{\xi} \eta) -{\eta}(\omega(\xi))+\omega(\nabla_{\eta}\xi).
\end{equation*}
However,
\begin{equation*}
\omega(\nabla_{\xi}\eta) - \omega(\nabla_{\eta}{\xi}) = \omega(\nabla_{\xi}\eta -\nabla_{\eta}{\xi}) = \omega([{\xi},\eta]),
\end{equation*}
and, since $d\omega(\eta, {\xi}) = 0$, $\omega([\eta, {\xi}]) = \eta(\omega({\xi})) -{\xi}(\omega(\eta))$. Therefore 
\begin{equation*}
(\nabla_{\xi}\omega)(\eta) - (\nabla_{\eta}\omega)({\xi}) = {\xi}(\omega(\eta))-{\eta}(\omega(\xi))+\omega([\eta, {\xi}])=0.
\end{equation*}
So, 
\begin{equation*}
\la{\xi, \nabla_{\eta} \xi}\ra = (\nabla_{\eta}\omega)(\xi) = (\nabla_{\xi}\omega)(\eta).
\end{equation*}
Now, since $\xi$ is tangent to $\partial M$ and $\omega(\eta)=0$ on $\partial M$,
\begin{equation*}
(\nabla_{\xi}\omega)(\eta) = \xi(\omega(\eta))-\omega(\nabla_{\xi}\eta) = \la{\nabla_{\xi}\xi, \eta}\ra = h^{\partial B}(\xi,\xi).
\end{equation*}
Hence,
\begin{equation*}
\la{\nabla_{\eta}\xi,\xi}\ra = \la{\xi, \nabla_{\eta} \xi}\ra =h^{\partial B}(\xi,\xi)=h^{\partial B}(N,N)\|\xi\|^2.
\end{equation*}
\end{proof}

\section{Proofs of Main Theorems}

\subsection{Eigenvalue Relationship}
\begin{thmnonum}{\upshape \textbf{\ref{ER}}} Let $M^n$ be an orientable free boundary minimal hypersurface of a convex body in $\mathbb{R}^{n+1}$ with Jacobi operator $J$. Then, for all positive integers $j$, one has that
\begin{equation*}
\lambda_j(J)\leq\lambda_{m(j)}(\Delta_1),
\end{equation*}
where $m(j) = {n+1\choose 2}(j-1)+1$ and $\lambda_{m(j)}(\Delta_1)$ is the $m(j)$th eigenvalue of the Laplacian eigenvalue problem with absolute boundary conditions.
\end{thmnonum}
\medskip

\begin{lem}\label{JC} For $\overline{V}, \overline{W}\in\overline{\mathcal{P}}$,  let 
\begin{equation*}
X_{V,W} = \la{\overline{V},N}\ra W - \la{\overline{W}, N}\ra V.
\end{equation*}
Let $\xi$ be any vector field on $M$ and consider the function $u=\la{X_{V,W}, \xi}\ra$. Then
\begin{equation*}
Ju = \la{X_{V,W},\Delta\xi}\ra +2v,
\end{equation*}
where $v$ is the smooth function
\begin{equation*}
v = \la{\nabla_{S(V)}\xi, W}\ra - \la{\nabla_{S(W)}\xi, V}\ra.
\end{equation*}
\end{lem}

\begin{proof}[Proof of Lemma \ref{JC}]
Since $u = \la{X_{V,W}, \xi}\ra = \la{\overline{V}, N}\ra\la{W, \xi}\ra - \la{\overline{W}, N}\ra\la{V, \xi}\ra$, from part (c) of Lemma \ref{LapIP}, (after some cancellations) we get that
\begin{equation*}
\Delta u = |S|^2u+\la{X_{V,W},\Delta\xi}\ra + 2v,
\end{equation*}
and so $Ju = \la{X_{V, W}, \Delta\xi}\ra+ 2v$.
\end{proof}

\begin{proof}[Proof of Theorem \ref{ER}]
Let $\{\phi_1, \phi_2,\ldots, \}$ be an orthonormal basis for $L^2(M)$ given by eigenfunctions of $J$, where $\phi_i$ is an eigenfunction associated to $\lambda_i(J)$. Let $V^m(\Delta_1) = \bigoplus_{i=1}^m E^N_{\lambda_i(\Delta_1)}$, where $E^N_{\lambda_i(\Delta_1)}$ is the space of eigenforms of $\Delta_1$ associated with $\lambda_1(\Delta_1)$ with absolute boundary conditions. We want to find $\omega\in V^m(\Delta_1)$, $\omega\not\equiv 0$, for which
\begin{equation}
\int_{M}\la{X_{V,W},\xi}\ra\phi_i dV = 0,\label{SOE}
\end{equation}
for $i=1,\ldots, j-1$ and for all $\overline{V}, \overline{W}\in\overline{\mathcal{P}}$, where $\xi$ is the vector field dual to $\omega$. Since $X_{V,W}$ is a skew-symmetric bilinear function of $\overline{V}, \overline{W}$, and since $\text{dim}\overline{\mathcal{P}} = \dim\mathbb{R}^{n+1}=n+1$,  there are $\binom{n+1}{2}$ equations that need to be satisfied in (\ref{SOE}) for each $i$, and therefore $\binom{n+1}{2} (j-1)$ homogeneous linear equations in total. So, if $m(j) = \binom{n+1}{2} (j-1)+1$, then we're guaranteed that there is a $\omega\in V^{m(j)}(\Delta_1)$, $\omega\not\equiv 0$, whose dual vector field satisfies (\ref{SOE}) for all $V, W$ and for $i=1,\ldots j-1$.
From the min-max principle and Lemma \ref{JC} we have that,
\begin{equation}\label{MMI}
\begin{split}
\lambda_j(J)\int_M u^2\,dV & \leq \int_M uJu\,dV + \int_{\partial M}\left(\frac{\partial u}{\partial \eta} +h^{\partial B}(N,N)u\right)u\,dA\\
& = \int_Mu\la{X_{V, W}, \Delta\xi}\ra\, dV+ 2\int_Muv\, dV +\int_{\partial M}\left(\frac{\partial u}{\partial \eta} +h^{\partial B}(N,N)u\right)u\,dA.\\
\end{split}
\end{equation}
In addition, 
\begin{equation*}
\begin{split}
\frac{\partial u}{\partial \eta} &= \eta\left(\la{\overline{V}, N}\ra\la{W, \xi}\ra - \la{\overline{W}, N}\ra\la{\overline{V}, \xi}\ra\right)\\
& = \la{\overline{V}, D_{\eta}N}\ra\la{W, \xi}\ra + \la{\overline{V}, N}\ra\left(\la{D_{\eta}\overline{W}, \xi}\ra + \la{\overline{W}, D_{\eta}\xi}\ra\right)\\
&\quad  - \la{\overline{W}, D_{\eta}N}\ra\la{V, \xi}\ra + \la{\overline{W}, N}\ra\left(\la{D_{\eta}\overline{V}, \xi}\ra + \la{\overline{V}, D_{\eta}\xi}\ra\right).\\
\end{split}
\end{equation*}

We'll now use an integration technique that exploits Lemma \ref{IC} to help us simplify (\ref{MMI}). We'll then apply Lemma \ref{BC} to get the claimed eigenvalue relationship.  

Using the product metric on $\overline{\mathcal{U}}\times \overline{\mathcal{U}}$, Lemma \ref{IC} implies that (pointwise)
\begin{equation*}
\begin{split}
&\int_{\overline{\mathcal{U}}\times \overline{\mathcal{U}}} u^2\,d\overline{V}d\overline{W}  = 2\|\xi\|^2,\\
&\int_{\overline{\mathcal{U}}\times \overline{\mathcal{U}}} u\la{X_{V,W}, \Delta\xi}\ra\, d\overline{V}d\overline{W}  = 2\la{\xi, \Delta\xi}\ra,\\
&\int_{\overline{\mathcal{U}}\times \overline{\mathcal{U}}} uv\, d\overline{V}d\overline{W}  = 0,\\
& \int_{\overline{\mathcal{U}}\times \overline{\mathcal{U}}} u\la{\overline{V}, D_{\eta}N}\ra\la{\overline{W}, \xi}\ra \,d\overline{V}d\overline{W} = 0,\\
& \int_{\overline{\mathcal{U}}\times \overline{\mathcal{U}}} u\la{\overline{V}, N}\ra \la{\overline{W}, D_{\eta} \xi}\ra\, d\overline{V}d\overline{W} = \la{\xi, D_{\eta}\xi}\ra = \frac{1}{2}\eta(\|\xi\|^2). 
\end{split}
\end{equation*}
Therefore, integrating (\ref{MMI}) over $\overline{\mathcal{U}}\times \overline{\mathcal{U}}$ yields
\begin{equation*}
2\lambda_j(J) \int_M\|\xi\|^2\, dV\leq 2\int_M\la{\xi, \Delta\xi}\ra\, dV + \int_{\partial M}\left( \eta(\|\xi\|^2) +2h^{\partial B}(N,N)\|\xi\|^2\right) dA.
\end{equation*}
From Lemma \ref{BC} we know that $\eta(\|\xi\|^2) =2h^{\partial B}(N,N)\|\xi\|^2$ on $\partial M$, since $\xi$ is the dual vector field of a 1-form satisfying the absolute boundary conditions. Moreover, since $\xi$ is the dual vector field to a linear combination of eigenforms of $\Delta_1$, it now follows that
\begin{equation*}
2\lambda_j(J)\int_M\|\xi\|^2\, dV\leq 2\lambda_{m(j)}(\Delta_1)\int_{M}\|\xi\|^2\, dV +4 \int_{\partial M}h^{\partial B}(N,N)\|\xi\|^2 \, dA.
\end{equation*}
Since $h^{\partial B}(U,U)<0$ for any vector tangent to $\partial B$, we get that
\begin{equation*}
2\lambda_j(J)\int_M\|\xi\|^2\, dV\leq 2\lambda_{m(j)}(\Delta_1)\int_{M}\|\xi\|^2\, dV.
\end{equation*}
Now, since $\omega\not\equiv0$, we can divide both sides by the $L^2(M)$-norm of $\xi$ to get 
\begin{equation*}
\lambda_j(J) \leq \lambda_{m(j)}(\Delta_1).
\end{equation*}
\end{proof}

\begin{rmk} We note that when $m(j)\leq\dim\mathcal{H}_N^1(M)$, \emph{i.e.} when $\omega$ is a linear combination of harmonic forms and therefore a harmonic form itself, we actually get the strict inequality $\lambda_j(J)<\lambda_{m(j)}(\Delta_1)=0$. This follows from the fact that $\omega\not\equiv 0$ implies that $\omega|_{\partial M}\not\equiv 0$ (see Theorem 3.4.4 on p.131 of \cite{Sch95}), and so we get the strict inequality $4\int_{\partial M} h^{\partial M}(N,N)\|\xi\|^2\,dA<0$.
\end{rmk}

\subsection{Index Bound} 

\medskip

\begin{thmnonum}{\upshape\textbf{\ref{IB}} (Index Bound)}  If $M$ is an orientable free boundary minimal hypersurface of a convex body in $\mathbb{R}^{n+1}$, then
{\upshape\begin{equation*}
\Ind(M) \geq \left \lfloor{\frac{\beta_a^1+{n+1\choose 2}-1}{{n+1\choose 2}}}\right \rfloor.
\end{equation*}}
\end{thmnonum}

\medskip

\begin{proof} Suppose $j$ is such that $m(j) \leq \dim\mathcal{H}_N^1(M):=\beta_a^1$. Then $\lambda_j(J)<\lambda_{m(j)}(\Delta)= 0$, so $\Ind(M)\geq j$. Now, $m(j) = {n+1\choose 2}(j-1)+1\leq \beta_a^1$, so $j\leq \left \lfloor{\frac{\beta_a^1+{n+1\choose 2}-1}{{n+1\choose 2}}}\right \rfloor$. Hence, $\Ind(M)\geq \left \lfloor{\frac{\beta_a^1+{n+1\choose 2}-1}{{n+1\choose 2}}}\right \rfloor$.
\end{proof}

\begin{cornonum}{\upshape\textbf{\ref{ICor}}}  If $M$ is an orientable free boundary minimal surface in a convex body in $\mathbb{R}^3$ with genus $g$ and $k$ boundary components, then
{\upshape\begin{equation*}
\Ind(M)\geq \left \lfloor{\frac{2g+k+1}{3}}\right \rfloor.
\end{equation*}} 
\end{cornonum}
\begin{proof} Since $\beta_a^1 = 2g+k-1$ for a surface (see Appendix \ref{DimCalc}), this follows directly from Theorem \ref{IB}.
\end{proof}

\begin{rmk} We note that Corollary \ref{ICor} can also be obtained by using the work of Ros. In \cite{Ro09}, Ros shows that if $\omega$ is a harmonic 1-form and $\xi$ is its dual vector field, then 
\begin{equation*}
\Delta\xi + \|A\|^2\xi = 2\la{\nabla\omega, A}\ra N,
\end{equation*}
and, if $\omega$ satisfies the absolute boundary conditions, then
\begin{equation*}
\la{\nabla_{\eta}\xi, \xi}\ra = h^{\partial B}(N,N)\|\xi\|^2.
\end{equation*}

 So, for $\xi = (\xi_1, \xi_2, \xi_3)$, if we use the notation $Q(\xi, \xi) = \sum_{i=1}^3Q(\xi_i, \xi_i)$. 
Now, assuming $\omega\not\equiv 0$,  
\begin{equation*}
\begin{split}
Q(\xi, \xi) & = -\int_M\la{\Delta\xi+\|A\|^2\xi, \xi}\ra dV + \int_{\partial M}(\la{\nabla_{\eta}\xi, \xi}\ra +h^{\partial B}(N,N)\|\xi\|^2) dA\\
& = 2\int_{\partial M}h^{\partial B}(N,N)\|\xi\|^2dA<0.
\end{split}
\end{equation*}
Hence $Q(X, X)<0$, and we get that $\dim\mathcal{H}_N^1(M)-3\cdot\Ind(M) = (2g+k-1)-3\cdot \Ind(M)\leq 0$, or $\Ind(M)\geq \lceil{\frac{(2g+k-1)}{3}}\rceil = \lfloor{\frac{2g+k+1}{3}}\rfloor$. 
\end{rmk}

\vspace{1cm}

\begin{appendix}
\section{The Dimension of the Space of Harmonic 1-Forms with Dirichlet Boundary Condition}\label{DimCalc}

It is well-known, we believe, that if $M$ is a surface with boundary $\partial M\neq\emptyset$, genus $g$ and $k$ boundary components, then $\dim \mathcal{H}_N^1(M ) = 2g + k -1$, but this result seems difficult to find in the literature. We give a proof here for completeness. When $M$ is a surface, it follows from Lefschetz duality that $\dim\mathcal{H}^1_N(M) = \dim\mathcal{H}^{2-1}_D(M)$, where $\mathcal{H}_D^1(M)$ is the space of harmonic 1-forms on $M$ which satisfy the \emph{relative} boundary conditions:
\begin{equation*}
i^*\omega = i^*\delta\omega =0,
\end{equation*}
where $i:\partial M \hookrightarrow M$ is the inclusion. So, to prove that $\dim \mathcal{H}_N^1(M ) = 2g + k -1$, we will show that $\dim \mathcal{H}_D^1(M ) = 2g + k -1$.

\begin{lem} Let $M$ be an orientable surface of genus $g$ with $k$ boundary components. Then $\dim\mathcal{H}^1_D(M) = 2g+k-1$.
\end{lem}

\begin{proof} Let $\mathcal{EH}_D^1(M)$ denote the subspace of harmonic fields with Dirichlet boundary conditions which are exact. Then,
\begin{equation*}
 \mathcal{H}_D^1(M) = \mathcal{EH}_D^1(M)\oplus \left(\mathcal{EH}_D^1(M)\right)^{\perp},
 \end{equation*}
 and $\dim\mathcal{H}^1_D(M) = \dim\mathcal{EH}_D^1(M) +\dim\left(\mathcal{EH}^1_D(M)\right)^{\perp}$. We claim that $\dim\mathcal{EH}_D^1(M) = k-1$ and $\dim\left(\mathcal{EH}^1_D(M)\right)^{\perp}=2g$.
 
 For the first claim, if $\omega \in \mathcal{EH}_D^1(M)$, then there is a function $u\in C^{\infty}(M)$ for which $\omega= du$. Since $\omega$ is a harmonic field with Dirichlet boundary conditions, it follows that $u$ is a harmonic function and is constant on the boundary. If we write the boundary as a disjoint union of $k$ curves, $\partial M = \Gamma_1\cup\dots\cup \Gamma_k$, then we get that $u|_{\Gamma_i}=c_i$, for some constant $c_i$, $i=1,\ldots k$. Now, the Dirichlet problem
 \begin{equation*}
\begin{cases}
\Delta u = 0\\
u|_{\Gamma_i} = c_i,
\end{cases}
\end{equation*}
has a unique solution for each choice of $(c_1,\ldots, c_k)$ (see pg. 307 of \cite{Ta96}). Let 
\begin{equation*}
\mathcal{F} = \left\{u\in C^{\infty}(M)\ \bigg| \ \Delta u = 0, u|_{\Gamma_i}=c_i, i=1\ldots k, \sum_{i=1}^k c_i = 0\right\}.
\end{equation*}
It easy to see that the differential $d|_{\mathcal{F}}:\mathcal{F}\rightarrow \mathcal{E}\mathcal{H}_D^1(M)$ is linear and bijective,  and so $\dim\mathcal{E}\mathcal{H}_D^1(M) = \dim\mathcal{F} = k-1$.

Let $\overline{M}$ be a smooth Riemannian manifold obtained from $M$ by gluing a disk into each of its boundary curves $\Gamma_i$. To prove the second claim, we will construct an isomorphism between $\left(\mathcal{EH}^1_D(M)\right)^{\perp}$ and  $H^1(\overline{M})$. The result will then follow from the fact that there are $2g$ cohomology classes of closed forms on $\overline{M}$.

Let $\theta \in \Omega(\overline{M})$ be a closed form. We'll first show that there is a closed form $\tilde{\omega}\in \Omega(\overline{M})$ supported on $M$ which is cohomologous to $\theta$. To see this, let $\tilde{D_i}$, $i=1,\ldots, k$, be a disk slightly larger than and containing $D_i$, and let $\phi_i$ be a smooth cut-off function for which $\phi_i|_{D_i}\equiv 1$ and $\phi_i|_{\overline{M}\setminus\tilde{D_i}}\equiv 0$. Since $\tilde{D_i}$ is simply-connected, $\theta|_{\tilde{D_i}} = df_i$ for some smooth functions $f_i$. Let $\tilde{\omega} = \theta - \sum_{i=1}^k d(\phi_i f_i)$. Then $\tilde{\omega}|_{D_i}\equiv 0$ and $d\tilde{\omega} =0$, so $\tilde{\omega}$ is a closed form in $\Omega(\overline{M})$ with compact support. Since $\sum_{i=1}^kd(\phi_i f_i)$ is exact, it follows that $\theta$ and $\tilde{\omega}$ are cohomologous. For simplicity, we will suppress the restriction notation and write $\tilde{\omega}|_M$ by $\tilde{\omega}$. Now, we claim that any closed form $\tilde{\omega}\in\Omega(M)$ with compact support is cohomologous to a form $\omega_0 \in (\mathcal{EH}_D(M))^{\perp}$. To see this, let $u$ be a solution to the Poisson problem
\begin{equation*}
\begin{cases}
\Delta u = -\delta \tilde{\omega}\\
u|_{\Gamma_i} = 0
\end{cases},
\end{equation*}
and define $\omega = \tilde{\omega}+du$. Then, $\omega$ is harmonic, since $\Delta\omega = \Delta\tilde{\omega}+\Delta du = d\delta\tilde{\omega} + 0 -d\Delta u=0$. Moreover, $i^*\omega = i^*\tilde{\omega}+d(i^*u) = 0$, so $\omega$ satisfies the Dirichlet boundary condition. Now, $\omega = \omega_0 + dv$ for some $\omega_0\in (\mathcal{EH}^1_D(M))^{\perp}$ and $dv \in\mathcal{EH}^1_D(M)$. Hence, $\omega_0$ is cohomologous to $\omega$, and therefore $\tilde{\omega}$ and $\theta$. Note that $\omega_0$ is unique, \textit{i.e.}, for any closed form $\theta\in \Omega(\overline{M})$, there is a unique $\omega_0\in(\mathcal{EH}_D^1(M))^{\perp}$ for which $\omega_0\sim\theta$. If $\omega_0^1,\omega_0^2\in(\mathcal{EH}_D^1(M))^{\perp}$ are two such forms, then $\omega_0^1\sim\theta\sim\omega_0^2$. Hence, $\omega_0^1-\omega_0^2 = d\zeta$, for some smooth function $\zeta$. However, $\omega_0^1-\omega_0^2\in(\mathcal{EH}_D^1(M))^{\perp}\subset(\mathcal{E}\Omega(M))^{\perp}$ and $d\zeta\in\mathcal{E}\Omega(M)$, so it follows that $\omega_0^1=\omega_0^2$.

Let $\mathcal{L}:H^1(\overline{M})\rightarrow (\mathcal{EH}_D^1(M))^{\perp}$ be the map $[\theta]\mapsto \omega_0$ (as above). Note that it follows from the uniqueness of $\omega_0$ that $\mathcal{L}$ is well-defined and linear.

Now, $\mathcal{L}$ is also injective. If $\mathcal{L}([\theta_1]) = \mathcal{L}([\theta_2])$, then $\theta_1+du_1=\theta_2+du_2$, for some smooth functions $u_1, u_2$, which yields $\theta_1\sim\theta_2$.

Finally, $\mathcal{L}$ is surjective. Suppose $\omega_0\in\left(\mathcal{EH}_D^1(M)\right)^{\perp}$. Then, since $i^*\omega_0\equiv0$, 
\begin{equation*}
\int_{\partial M}\omega_0 = 0,
\end{equation*}
and it follows that $\omega_0$ is exact in a neighbourhood of each boundary curve, \textit{i.e.}, $\omega_0 = d\psi_i$ in a neighbourhood of $\Gamma_i$. Since we can extend each $\psi_i$ smoothly over $D_i$, we can extend $\omega_0$ to a closed form $\theta\in\overline{M}$. It follows from the well-definedness of $\mathcal{L}$ that $\mathcal{L}$ does not depend on the choice of $\tilde{D}_i$ or $\phi_i$, $i=1,\ldots k$. Hence, $\mathcal{L}([\theta]) = \omega_0$. 
\end{proof}

\end{appendix}


\begin{thebibliography}{99}
\bibitem{FPZ15}
A.\ Folha, F.\ Pacard, and T.\ Zolotareva.
\textit{Free boundary minimal surfaces in the unit 3-ball}, arXiv:1502.06812 [math.DG].

\bibitem{Fr07}
A.\ Fraser.
\textit{Index estimates for minimal surfaces and k-convexity}, Proc.\ Amer.\ Math.\ Soc., \textbf{135} (2007), no.\ 11, 3733-3744. 

\bibitem{FS11}
A.\ Fraser and R.\ Schoen.
\textit{The first steklov eigenvalue, conformal geometry, and minimal surfaces}, Adv.\ Math.\ \textbf{226} (2011), no.\ 5, 4011-4030.

\bibitem{FS15} A.\ Fraser and R.\ Schoen.
\textit{Sharp eigenvalue bounds and minimal surfaces in the ball}, Invent.\ Math.\ \textbf{203} (2016), no.\ 3, 823-890.


\bibitem{FS15b} A.\ Fraser and R.\ Schoen. 
\textit{Uniqueness theorems for free boundary minimal disks in space forms}, Int.\ Math.\ Res.\ Not.\ IMRN 2015, no.\  17, 8268-8274.

\bibitem{GJ86} M.\ Gr\"uter and J.\ Jost.
\textit{On embedded minimal disks in convex bodies},  Ann.\ Inst.\ H.\ Poincar\'e Anal.\ Non Lin\'eare \textbf{3} (1986), no. 5, 345-390.

\bibitem{Jo89}
J.\ Jost.
\textit{Embedded minimal surfaces in manifolds diffeomorphic to the three-dimensional ball or sphere}, J.\ Differential Geom., \textbf{30} (1989), no. 2, 555-577.

\bibitem{Ma12}
F.\ Marques and A.\ Neves.
\textit{Min-max theory and the Willmore conjecture}, arXiv:1202.6036 [math.DG].

\bibitem{MNS}
D.\ Maximo, I.\ Nunes and G.\ Smith.
\textit{Free boundary minimal annuli in convex three-manifolds}, to appear in J.\ Differential Geom.

\bibitem{Ni85}
J.\ C.\ C.\ Nitsche.
\textit{Stationary partitioning of convex bodies}, Arch.\ Rational Mech.\ Anal., \textbf{89} (1985) , no. 1, 1-19.

\bibitem{Ro09}
A.\ Ros.
\textit{Stability of minimal and constant mean curvature surfaces with free boundary}, Mat.\ Contemp., \textbf{35} (2008), 221-240.

\bibitem{Sav10}
A.\ Savo.
\textit{Index bounds for minimal hypersurfaces of the sphere}, Indiana Univ.\ Math.\ J., \textbf{59} (2010), no.\ 3, 823-837.

\bibitem{Sch95}
G.\ Schwarz.
\textit{Hodge Decomposition -- A Method for Solving Boundary Value Problems},
Lecture Notes in Mathematics \textbf{1607}, Springer-Verlag, New York, 1995.

\bibitem{St84}
B.\ Struwe.
\textit{On a free boundary problem for minimal surfaces},
Invent.\ Math., \textbf{75} (1984), 547-560.

\bibitem{Ta96}
M.\ Taylor.
\textit{Partial Differential Equations I: Basic Theory}, Applied Mathematical Sciences \textbf{115}, Springer-Verlag, New York, 1996.

\bibitem{Ur90} F.\ Urbano,
\textit{Minimal surfaces with low index in the three-dimensional sphere}, Proc.\ Amer.\ Math.\ Soc.,\textbf{108} (1990), no.\ 4, 989-992.
\end{thebibliography}
\end{document}